\documentclass[11pt,a4paper,oneside]{amsart}

\usepackage{amsmath,amssymb,amscd,amsthm,esint}

\usepackage{graphics,amsmath,amssymb,amsthm,mathrsfs,amsfonts,amsrefs}

\usepackage{sidecap}
\usepackage{float}
\usepackage{extarrows}
\usepackage{booktabs}
\usepackage{verbatim}
\usepackage{hyperref}
\usepackage[usenames,dvipsnames]{xcolor}

\newtheorem{thm}{Theorem}[section]
\newtheorem{lemma}[thm]{Lemma}
\newtheorem{cor}[thm]{Corollary}
\newtheorem{prop}[thm]{Proposition}

\theoremstyle{definition}
\newtheorem{remark}[thm]{Remark}
\newtheorem{assumption}[thm]{Assumption}
\newtheorem{definition}[thm]{Definition}

\def\XXint#1#2#3{{\setbox0=\hbox{$#1{#2#3}{\int}$}
         \vcenter{\hbox{$#2#3$}}\kern-.5\wd0}}

\def\R{\mathbb{R}}

\DeclareMathOperator{\supp}{supp}

\numberwithin{equation}{section}

\begin{document}

\title[ Inverse Scattering Problems of NLS]{High-Velocity Inverse Scattering\\ for Nonlinear Schr\"odinger Equations\\ with Spatially Dependent Nonlinearities}
	\author[S. Masaki]{Satoshi Masaki}
\address{Department of mathematics, 
Hokkaido University, Sapporo Hokkaido, 060-0810, Japan}
\email{masaki@math.sci.hokudai.ac.jp}
	\author[Y. Zhao]{Yaxin Zhao}
\address{Department of Mathematics, University of Connecticut, Storrs, CT 06269, USA}	\email{cql25001@uconn.edu}
	

	\begin{abstract}
We study a high-velocity inverse scattering problem for nonlinear Schrödinger equations with spatially dependent nonlinearities in dimensions $d\ge3$. We consider the whole mass-supercritical and energy-subcritical range, including the endpoint cases. By introducing a moving frame adapted to highly boosted initial data, we construct the scattering operator for a class of large incoming states generated by Galilean boosts. The key observation is that, although the boosted data become large in Sobolev norms, the nonlinear interaction becomes effectively weak at high velocity due to rapid spatial separation.
Using the resulting high-velocity asymptotics, we derive a reconstruction formula for the X-ray transform of the coefficient. As a consequence, we prove that the scattering operator uniquely determines both the nonlinearity exponent and the spatial coefficient. 
Our results extend previous work of Watanabe to all dimensions $d \ge 3$, include the endpoint nonlinearities, and replace the repulsiveness and radial monotonicity assumptions on the coefficient by suitable decay conditions.
\end{abstract}

        
\keywords{nonlinear Schrödinger equation, inverse scattering, high-velocity scattering, X-ray transform, scattering operator}

\subjclass[2020]{Primary 35P25, Secondary 35Q55, 35R30, 35B40}

\maketitle	

\section{Introduction}\label{sec-intro}

\subsection{Background and main idea}

Inverse scattering problems for nonlinear dispersive equations aim to recover information on nonlinear interactions from the associated scattering operator. In recent years, various inverse problems for nonlinear Schr\"odinger equations (NLS) have been studied extensively, motivated both by mathematical scattering theory and by applications to wave propagation in inhomogeneous media. In this paper, we consider the nonlinear Schrödinger equation
\begin{align}\label{E:NLS}
(i\partial_t+\Delta)u = \alpha(x)|u|^p u,
\qquad
(t,x)\in\R\times\R^d,
\end{align}
where $d\ge3$ is the spatial dimension, $\alpha(x)$ is a spatially dependent coefficient, and $p$ belongs to the whole mass-supercritical and energy-subcritical range including both endpoint cases:
\begin{align}\label{E:range_p}
\tfrac{4}{d}\le p\le \tfrac{4}{d-2}.
\end{align}

The study of scattering operators for nonlinear dispersive equations goes back to the pioneering work of Strauss \cite{St74}. Since then, inverse scattering problems for nonlinear Schrödinger equations have been extensively studied in various settings; see for instance Weder \cite{We97,We01}. More recent developments include \cite{CM,KMV23,KMV25,Sa08,Sa12} and the references therein.

From the viewpoint of high-velocity inverse scattering, Watanabe \cite{Wa} studied the three-dimensional case of \eqref{E:NLS} in the non-endpoint range of \eqref{E:range_p} under the repulsive condition $\alpha\ge0$ and the radial monotonicity condition $x\cdot \nabla \alpha \le 0$, by adapting the geometrical high-velocity method of Enss and Weder \cite{EW}. In particular, he considered the injectivity of the map sending the pair $(p,\alpha)$ to the scattering operator $S$.
Later, Murphy \cite{Mu} pointed out that part of the analysis may require additional justification and established the result in the mass-subcritical and mass-critical cases.

Our approach further develops this geometric high-velocity framework in the energy-space setting including endpoint nonlinearities. 
In the present paper, we revisit the high-velocity analysis of Watanabe \cite{Wa} and establish the result throughout the full range \eqref{E:range_p}.
In particular,
we extend the result to all dimensions $d\ge3$, including the endpoint cases. Furthermore, the repulsiveness condition $\alpha\ge0$ and the radial monotonicity condition $x\cdot\nabla\alpha\le0$ are replaced by a suitable decay assumption on $\alpha$.

The basic idea is to extract information on the X-ray transform of the coefficient,
\[
\int_{-\infty}^{\infty}\alpha(x+2\theta t)\,dt,
\qquad
\theta\in\mathbb{S}^{d-1},
\]
by considering highly boosted incoming states of the form
\[
u_-=e^{iv\cdot x}\phi,
\]
where $v$ is sufficiently large and $\theta=v/|v|$. Roughly speaking, the boosted wave packet propagates almost freely along the Galilean trajectory $x+2tv$. Since the coefficient $\alpha$ is spatially localized, the interaction between the wave packet and the coefficient becomes effectively short-lived in the moving frame. As a consequence, the leading high-velocity asymptotics naturally produce the X-ray transform of the coefficient.

Let us explain the problem in more detail and describe the main difficulty. We consider solutions to \eqref{E:NLS} satisfying
\[
e^{-it\Delta}u(t)-u_-\to0
\qquad
(t\to-\infty),
\]
where $u_-$ is a highly boosted state. If the corresponding solution exists globally forward in time and satisfies
\[
e^{-it\Delta}u(t)-u_+\to0
\qquad
(t\to\infty),
\]
then the scattering operator $S$ is defined by assigning $u_-$ to $u_+$.

The principal difficulty is that highly boosted data are not small in $H^1$. Indeed, for any nontrivial $\phi$,
\[
\|e^{iv\cdot x}\phi\|_{H^1}\to\infty
\qquad
(|v|\to\infty).
\]
For linear equations, the largeness of the data is not problematic. In nonlinear problems, however, the relative strength of the nonlinearity depends essentially on the size of the solution. In particular, when $\alpha$ is constant, the equation becomes translation invariant, and depending on the sign of the coefficient, standing-wave solutions or blow-up solutions may occur. Thus, the issue of large data is essential in the study of the well-definedness of the scattering operator.

In the present setting, however, the mechanism producing large $H^1$-norms is completely explicit: it arises solely from the Galilean boost. Although the $H^1$-norm becomes large, this largeness is, in a sense, only apparent once one identifies its origin. The perturbative structure in our problem therefore arises not from small amplitude, but from rapid spatial separation induced by the Galilean flow.

Our key idea is to introduce a moving frame in which the boost disappears. In this moving coordinate system, the solution appears almost stationary, while the coefficient $\alpha$ moves rapidly along the opposite direction. Since the coefficient is spatially localized, this rapid motion weakens the effective nonlinear interaction. Consequently, from the viewpoint of the moving frame, the boost is no longer a harmful mechanism producing largeness, but rather a favorable mechanism producing effective smallness. Thus, although the initial data are large in Sobolev norms, the nonlinear interaction itself becomes weak at high velocity.

Using this idea, we construct the scattering operator for highly boosted data. Once the scattering operator is constructed, arguments similar to those in \cite{Wa,Mu} show that the scattering operator uniquely determines both the exponent $p$ and the coefficient $\alpha$.

\subsection{Main results}

Given an integer $a\ge0$ and a real number $s\in[0,d/2)$, we define the space $X^{a,s}$ by the norm
\[
\|\varphi\|_{X^{a,s}}
=
\sum_{i\le a}
\||\nabla|^i\varphi\|_{X^s},
\]
where the space $X^s$ is defined by
\[
\|\varphi\|_{X^s}
:=
\|\langle x\rangle^s\varphi\|_{L^2}
+
\||\nabla|^s\varphi\|_{L^{\frac{2d}{d+2s}}}.
\]
The spaces $X^{a,s}$ are introduced to measure both spatial localization and weighted regularity adapted to the moving-frame analysis.

Our assumptions on the coefficient $\alpha$ are as follows.

\begin{assumption}\label{A:alpha}
Let $p$ satisfy \eqref{E:range_p}. We introduce the following conditions on $\alpha:\mathbb{R}^d\to\mathbb{C}$:
\begin{itemize}
\item[$(\mathrm{A1})$]
There exists $\sigma_1>\frac{p+1}{p+2}$ such that
$
\langle x\rangle^{\sigma_1}\alpha\in W^{1,\infty}$.
\item[$(\mathrm{A2})$]
There exists $\sigma_2\in(1,d/2)$ such that
$
\langle x\rangle^{\sigma_2}\alpha\in L^\infty$.
\end{itemize}
\end{assumption}
Condition $(\mathrm{A1})$ is used to construct the scattering operator for highly boosted data, while $(\mathrm{A2})$ is required in the analysis of the inverse problem and the recovery formula.
Our restriction $d\ge3$ comes from this point.

To formulate the assumptions on the data, we introduce two indices $s_1$ and $s_2$.

\begin{definition}\label{D:sj}
Suppose $p$ satisfies \eqref{E:range_p}. We define $s_1,s_2\in(0,d/2)$ as follows.
\begin{enumerate}
\item
Suppose $\alpha$ satisfies $(\mathrm{A1})$ with some $\sigma_1>\frac{p+1}{p+2}$. Let $s_1\in(0,d/2)$ satisfy
\[
ps_1\in\left(1,\tfrac{p+2}{p+1}\sigma_1\right].
\]
\item
Suppose $\alpha$ satisfies $(\mathrm{A2})$ with some $\sigma_2>1$. We define
\[
s_2=\sigma_2.
\]
Since $\sigma_2<d/2$, we also have $s_2<d/2$.
\end{enumerate}
\end{definition}

We denote by $S=S_{p,\alpha}$ the scattering map sending $u_-\in H^1$ to $u_+\in H^1$, namely,
\begin{align}
Su_-
=
u_-
-i\int_{\R}
e^{-it\Delta}\alpha(x)|u|^pu(t)\,dt,
\end{align}
where $u$ is the solution to \eqref{E:NLS} scattering to $u_-$ as $t\to-\infty$.

We define
\begin{equation}\label{E:defB}
B=B_C
:=
\bigcup_{M>0}
\left\{
e^{iv\cdot x}\varphi
:
\|\varphi\|_{H^2\cap X^{1,s_1}}\le M,
\quad
|v|\ge CM^{p+2}+1
\right\},
\end{equation}
where $C>0$ is a constant. The set $B_C$ consists of highly boosted data whose velocity is sufficiently large compared with the size of the profile.

Our first main result establishes the well-definedness of the scattering operator on $B_C$.

\begin{thm}\label{T:well-defined}
Let $d\ge3$, and let $p$ satisfy \eqref{E:range_p}. Suppose that $\alpha$ satisfies $(\mathrm{A1})$. Let $s_1$ be given by Definition \ref{D:sj}. Then there exists a positive constant $C>0$, depending only on $d$, $p$, and
\[
\|\langle x\rangle^{\sigma_1}\alpha\|_{W^{1,\infty}},
\]
such that the scattering operator $S$ is well-defined as a map from $B_C$ to $H^1$.
\end{thm}

\begin{remark}
Compared with the result of Watanabe \cite{Wa}, we treat general dimensions $d\ge3$ and the full range \eqref{E:range_p}, including the endpoint cases. In addition, the assumptions $\alpha \ge0$ and $x\cdot \nabla \alpha \le 0$ imposed in \cite{Wa} are replaced by the decay condition $(\mathrm{A1})$, allowing $\alpha$ to be a general complex-valued function.
\end{remark}

We now turn to the inverse problem.
The following theorem identifies the leading high-velocity asymptotics of the scattering operator. The main contribution is given by the nonlinear interaction accumulated along the Galilean trajectory of the boosted wave packet, which results in the X-ray transform of the coefficient.

\begin{thm}\label{T:inverse-formula}
Let $d\ge3$ and let $p$ satisfy \eqref{E:range_p}. Suppose that $\alpha$ satisfies $(\mathrm{A1})$ and $(\mathrm{A2})$. 
 Let $s_1$ and $s_2$ be given by Definition \ref{D:sj}.
Let
\[
\varphi,\psi\in H^2\cap X^{1,s_1}\cap X^{s_2}.
\]
Then, for every $\theta\in\mathbb{S}^{d-1}$,
\begin{align}\label{E:recovery}
\lim_{\rho\to\infty}
i\rho
\left\langle
(S-I)(e^{i\rho\theta\cdot x}\varphi),
e^{i\rho\theta\cdot x}\psi
\right\rangle
=
\int_{\R}
\left\langle
\alpha(x+2t\theta)
|\varphi|^p\varphi,
\psi
\right\rangle
dt,
\end{align}
where $S:B\to H^1$ denotes the scattering map constructed in Theorem \ref{T:well-defined}.
\end{thm}

\begin{remark}
Theorem \ref{T:inverse-formula} extends \cite[Theorem 3.1]{Wa} by treating general dimensions $d\ge3$ and including the endpoint cases. 
A justification for part of the argument in \cite[Theorem 3.1]{Wa} was discussed in \cite{Mu}. In the present paper, we establish the corresponding result by introducing a moving frame adapted to the high-velocity regime and extend it to all dimensions $d\ge3$, including the endpoint cases. As for the assumptions on $\alpha$, apart from the shared assumption $(\mathrm{A2})$, we replace the assumptions $\alpha\in W^{2,\infty}$, $\alpha\ge0$, and $x\cdot\nabla\alpha\le0$ in \cite{Wa} with the decay condition $(\mathrm{A1})$.
\end{remark}

Using this formula, we can reconstruct both the exponent $p$ and the X-ray transform of $\alpha$ from the scattering map.

\begin{thm}\label{T:recoveryf}
Let $d\ge3$ and let $p$ satisfy \eqref{E:range_p}. Suppose that $\alpha$ satisfies $(\mathrm{A1})$ and $(\mathrm{A2})$. 
 Let $s_1$ and $s_2$ be given by Definition \ref{D:sj}.
Let $\varphi\in H^2\cap X^{1,s_1}\cap X^{s_2}$
be a function.
If $\varphi$ and $\theta \in \mathbb{S}^{d-1}$ satisfy
\[
	\int_\R \int_{\R^d} \alpha (x+2t\theta) |\varphi(x)|^{p+2} dxdt \neq0
\]
then 
\begin{equation}\label{E:precov}
p=-1+
\log
\left(
\lim_{\rho\to\infty}
\frac{
\left\langle
(S-I)(e^{1+i\rho\theta\cdot x}\varphi),
e^{i\rho\theta\cdot x}\varphi
\right\rangle
}{
\left\langle
(S-I)(e^{i\rho\theta\cdot x}\varphi),
e^{i\rho\theta\cdot x}\varphi
\right\rangle
}
\right).
\end{equation}
Moreover, if $\varphi$ is compactly supported and normalized as $\|\varphi\|_{L^{p+2}(\R^d)}=1$,
then for any $\theta\in\mathbb{S}^{d-1}$ and $y\in\R^d$,
\begin{equation}\label{E:arecov}
\int_{\R}
\alpha(y+2\theta t)\,dt
=
\lim_{n\to\infty}
\lim_{\rho\to\infty}
i\rho
\left\langle
(S-I)(e^{i\rho\theta\cdot x}\varphi_n(x-y)),
e^{i\rho\theta\cdot x}\varphi_n(x-y)
\right\rangle,
\end{equation}
where $\varphi_n(x)=n^{\frac{d}{p+2}}\varphi(nx)$.
\end{thm}

As an immediate consequence, we obtain the injectivity of the map
\[
(p,\alpha)\mapsto S_{p,\alpha}.
\]
The uniqueness of $\alpha$ follows from the injectivity of the X-ray transform, which in turn follows from the Fourier slice theorem (see for example Helgason \cite{He99}).

\begin{cor}
Let $d\ge3$ and suppose that $(p,\alpha)$ and $(\bar p,\bar\alpha)$ satisfy \eqref{E:range_p} together with assumptions $(\mathrm{A1})$ and $(\mathrm{A2})$. Let $S:B_{C_1}\to H^1$ and $\bar S:B_{C_2}\to H^1$ denote the corresponding scattering maps. If
\[
S(f)=\bar S(f)
\qquad
\text{for all }f\in B_{\max (C_1,C_2)} \cap C_0^\infty,
\]
then
\[
p=\bar p,
\qquad
\alpha=\bar\alpha.
\]
\end{cor}

Since the Fourier slice theorem provides an explicit inversion formula for the X-ray transform, the coefficient $\alpha$ can in principle be reconstructed explicitly from the high-velocity scattering data. Since this reconstruction procedure is classical, we do not pursue the explicit formula here.

\medskip

The rest of the paper is organized as follows.
In Section \ref{sec-pre}, we introduce the function spaces and recall several basic estimates.
In Section \ref{sec-direct}, we prove Theorem \ref{T:well-defined}.
Finally, in Section \ref{sec-inverse}, we prove Theorems \ref{T:inverse-formula} and \ref{T:recoveryf}.

\section{Preliminaries}\label{sec-pre}
\numberwithin{equation}{section}

We first introduce some notations which will be used throughout this paper.
 
We sometimes use $ A\lesssim B$ to denote the inequality $A\le CB$ for some positive constant $C>0$ and write $\langle x \rangle=\sqrt{1+\left | x \right |^2 } $. The standard space-time Lebesgue spaces is defined by \[
\left \| u \right \| _{L_t^q L_x^r(\mathbb{R}\times \mathbb{R}^d  )}=\left \| \left \| u(t,x) \right \|_{L_x^r(\mathbb{R}^d )}  \right \|_{L_t^q(\mathbb{R} )} .
\] We denote $\left \| \cdot \right \| _{L_t^q L_x^r(\mathbb{R}\times \mathbb{R}^d  )}$ by $\left \| \cdot \right \| _{L_t^q L_x^r}$ and assume that all space-time norms will be taken over $\mathbb{R}\times \mathbb{R}^d$ unless indicated otherwise. 
We will work with the pair
\[
(q_0,r_0)=\left(p+2,\frac{2d(p+2)}{d(p+2)-4}\right),
\]
which is an admissible pair under \eqref{E:range_p}. 
In particular, $2<r_0<p+2\le \frac{2d}{d-2}$.

The free Schrödinger group is denoted by $e^{it\Delta}$. From the definition of $e^{it\Delta}$, we obtain useful dispersive estimates and an identity as follows; the latter one is also known as the Galilean transform, which is one of the essential steps in the proof of the scattering theory and recovery formulas from the large scattering data of \eqref{E:NLS}.
\begin{prop}
    If $2\le p\le \infty$ and $t\ne 0$, then \begin{align}
        \left \| e^{i t\Delta} \varphi  \right \| _{L^{\infty}}&\lesssim \left | t \right |^{-\frac{d}{2} }\left \| \varphi \right \|_{L^1},\label{dispersive1}\\
        \left \| e^{i t\Delta} \varphi  \right \| _{L^{2}}&= \left \| \varphi \right \|_{L^2},\label{dispersive2}\\
        \left \| e^{i t\Delta} \varphi  \right \| _{L^{p}}&\lesssim \left | t \right |^{-(\frac{d}{2}-\frac{d}{p}) }\left \| \varphi \right \|_{L^{p'}}.\label{dispersive3}
    \end{align} where $p'$ denotes the dual Hölder exponent of $p$.
\end{prop}
\begin{prop}[Galilean transform]\label{Galilean transform}
    For any $v\in \mathbb{R}^d$, we have\begin{align}\label{GT}
        [ e^{it\Delta}e^{iv\cdot x}\varphi] (x)=e^{-i\left | v \right |^2 t }e^{iv\cdot x}[e^{it\Delta}\varphi](x-2tv)
    \end{align}
\end{prop}

We say $(q,r)$ an admissbile pair if $\frac{2}{q}=\frac{d}{2}-\frac{d}{r}$ and $(q,r,d)\neq(2,\infty,2)$. 
In order to estimate the nonlinearity of the NLS equation, we will make use of the following Strichartz's estimate and the admissible pair $(q,r)$ for the free Schrödinger group.

\begin{thm}[Strichartz's estimate]\label{Strichartz's estimate}
    For every $\varphi\in L^2$ and admissible pair $(q,r)$, we have
    \begin{align}
        \left \| e^{it\Delta}\varphi \right \| _{L_t^q L_x^r}\lesssim \left \| \varphi \right \| _{L^2}.
    \end{align}
    
For every admissible pair $(q,r)$ and $(\bar q,\bar r)$ and $F\in L_{t}^{\bar q'}L_x^{\bar r'}$, we have \begin{align}
        \left \| \int_{-\infty}^{t} e^{i(t-s)\Delta}F(s) ds \right \| _{L_t^qL_x^{r}}\lesssim \left \| F \right \|_{L_t^{^{\bar q'}}L_x^{\bar r'}},
    \end{align}
    where $(\bar q',\bar r')$ denotes the dual Hölder exponent of $(\bar q, \bar r)$.
\end{thm}

\subsection{Mismatch-type estimates}
The following lemmas will play a key role in estimating large solutions of NLS equations by establishing an appropriate weighted estimate for solutions of the free Schrödinger equation. Specifically, we will utilize the mismatch-type estimate to provide sufficient ``smallness," guaranteeing the highly boosted large data could be well estimated in some weighted space.
We recall the estimate from \cite{Mu}.

\begin{prop}\label{weighted estimate-free solu}
    Let $q: \mathbb{R}^d\to \mathbb{C}$ satisfy $|q(x)|\lesssim \left \langle x \right \rangle ^{-s}$ for some $s\in \left (  0,\frac{d}{2}   \right ) $. Then,
\begin{align*}
	\left \| q(x)e^{it\Delta}e^{iv\cdot x}\varphi \right \|_{L^2}
	=\left \| q(x + 2tv)e^{it\Delta}\varphi \right \|_{L^2}
	\lesssim \left \langle tv \right \rangle^{-s} \left \| \varphi \right \|_{X^s} 
\end{align*}
holds for any $v\in \mathbb{R}^d$, $\varphi \in X^s$, and
    $t\in \mathbb{R}$.
\end{prop}

Now, we proceed to extend this proposition for solutions of the free Schrödinger equation to $H^{1,r_0}_x$ space by imposing higher decay assumption on the function $q$. This is one of main steps in the proof of Theorems \ref{T:well-defined} and  \ref{T:inverse-formula}.

\begin{lemma}\label{weigheted in t,x-free solu}
Suppose that $p$ satisfies \eqref{E:range_p}.
Let $s\in (0,d/2)$ 
satisfy $ps>1$.
Then, it holds for any $\varphi\in H^2 \cap X^{1,s}$ that
\begin{align}
         \| \langle x+2tv \rangle ^{-\frac{p}{p+2}s}e^{it\Delta}\varphi  \|_{L_t^{q_0} H_x^{1,r_0}}\lesssim |v|^{-\frac{1}{p+2} }\left \| \varphi \right \|_{H^2 \cap X^{1,s}}  .
    \end{align}
\end{lemma}

\begin{proof}
Let $c=\frac{p}{p+2}\in (0,1)$.
We begin with the estimate with respect to the space variable. 
   Recall that $2<r_0<p+2\le \frac{2d}{d-2}$. 
    Hence, by Hölder's inequality, we have \begin{align*}
        \left \| \langle x+2tv \rangle ^{-cs} e^{it\Delta}\varphi \right \| _{L^{r_0}}
&=\left \| [\langle x+2tv \rangle ^{-s} e^{it\Delta}\varphi]^c [e^{it\Delta}\varphi]^{1-c}\right \|_{L^{r_0}}\\
&\lesssim \left \| \langle x+2tv \rangle ^{-s} e^{it\Delta}\varphi \right \| _{L^2}^{c}\left \| e^{it\Delta}\varphi \right \|^{1-c}_{L^{\frac{2d}{d-2} }} \\
&\lesssim \left \langle tv \right \rangle^{-cs}\left \| \varphi \right \|_{X^s}^c\left \| \nabla \varphi \right \|_{L^2}^{1-c},  
    \end{align*}where we have used Proposition \ref{weighted estimate-free solu} and Sobolev's embedding for the last step.
Similarly, noting that $[\nabla, \langle x+2tv \rangle ^{-cs}]$ is controlled by $\langle x+2tv \rangle ^{-cs}$, we see that
        \begin{align*}
            \left \| \langle x+2tv \rangle ^{-cs} e^{it\Delta}\varphi \right \| _{H_x^{1,r_0}}&\lesssim \left \langle tv \right \rangle^{-cs}\left \| \varphi \right \|_{X^s}^c\left \| \nabla \varphi \right \|_{L^2}^{1-c}+\left \langle tv \right \rangle^{-cs}\left \| \nabla\varphi \right \|_{X^s}^c\left \| |\nabla|^2 \varphi \right \|_{L^2}^{1-c} \\
&\lesssim\left \langle tv \right \rangle^{-cs}\left \| \varphi \right \|_{H^2 \cap X^{1,s}}.
        \end{align*}
Take the $L^{q_0}_t$-norm of both sides to obtain
\begin{align*}
     \left \| \left \langle x+2 tv \right \rangle^{-cs} e^{it\Delta}\varphi \right \| _{L_t^{q_0}H_x^{1,r_0}}\lesssim\left \| \left \langle tv \right \rangle^{-cs} \right \| _{L^{q_0}_t}\left \| \varphi \right \|_{H^2 \cap X^{1,s}}\lesssim |v|^{-\frac{1}{p+2}}\left \| \varphi \right \|_{H^2 \cap X^{1,s}}.
\end{align*} 
The last step requires the condition $csq_0=ps>1$. 
This completes the proof.
\end{proof}


\section{The Direct Problem}\label{sec-direct}

In this section, we consider the scattering problem of \eqref{E:NLS} in the intercritical regime. 
\begin{thm}\label{wellposedness-u}
Let $d\ge 3$ and
let $p$ satisfy \eqref{E:range_p}.
Suppose that $\alpha$ satisfies $(\mathrm{A1})$.
Let $s_1$ be the number given in Definition \ref{D:sj}.
Let $\varphi\in H^2 \cap X^{1,s_1}$
and $v\in \R^d$ and set $u_{-}=e^{iv\cdot x}\varphi$.
There exists a constant $C>0$ such that 
if 
\[
	| v  | \ge C \| \varphi \|_{H^2 \cap X^{1,s_1}}^{p+2} +1
\]
then there exists a unique global solution $u$ to \eqref{E:NLS} as well as the scattering state $u_{+}\in H^1$ satisfying
\begin{align}\label{u-property}
  \|\langle x\rangle^{-\frac{p}{p+2}s_1} u\|_{L^{q_0}_t H^{1,r_0}_x ( \mathbb{R} \times \mathbb{R}^d)} \lesssim |v|^{\frac{p+1}{p+2} }\left \| \varphi \right \|_{H^2 \cap X^{1,s_1}} \end{align}
and
\begin{equation}
\lim_{t \to \pm \infty} \| u(t) - e^{it\Delta} u_\pm \|_{H^1} = 0.
\end{equation}
In particular, the scattering map $S:u_- \mapsto u_+$ is defined as a map from $B_C$ to $H^1$.
\end{thm}

\subsection{Moving frame}
The difficulty in proving Theorem \ref{wellposedness-u} is that $e^{iv\cdot x} \varphi$ is large in $H^1$
for large $|v|$, even if $\varphi$ is small. 
By introducing the suitable moving frame, this largeness factor can be removed.
Indeed, define a new variable $w=w(t,x)$ by
\begin{align}\label{def-w}
        w(t)= T(-2tv)e^{i\left | v \right |^2 t }e^{-iv\cdot x}u(t),
\end{align}
where $(T(v)u)(x)=u(x-v)$ is a translation operator. 
Then, the equation for the new variable takes the form
\begin{equation}\label{eqn-w}
	w(t) = e^{it\Delta} \varphi- i \int_{-\infty}^{t} e^{i(t-s)\Delta} \alpha(x+2sv) |w|^p w(s) \, ds.
\end{equation}
Note that
the factor $e^{iv\cdot x}$ is removed in the new formulation.
Furthermore, by means of \eqref{GT}, one has
\[
	e^{-it\Delta} w(t) =  e^{-iv\cdot x} e^{-it\Delta} u(t). 
\]

Hence, $u(t)$ scatters to $u_\pm$
as $t\to \pm \infty$ if and only if $w(t)$ scatters to $e^{-iv\cdot x} u_\pm$ as $t\to \pm \infty$.
We shall show the following.
\begin{thm}\label{well-posedness-w}
Let $d\ge 3$ and
let $p$ satisfy \eqref{E:range_p}.
Suppose that $\alpha$ satisfies $(\mathrm{A1})$.
Let $s_1$ be the number given in Definition \ref{D:sj}.
Let $\varphi\in H^2 \cap X^{1,s_1}$
and $v\in \R^d$ and set $w_-=\varphi$.
There exists a constant $C>0$ such that 
if 
\[
	| v  | \ge C \| \varphi \|_{H^2 \cap X^{1,s_1}}^{p+2} +1
\]
then there exists a unique global solution $w$ to \eqref{eqn-w} as well as the scattering state $w_{+}\in H^1$ satisfying
\begin{align}\label{w-property}
  \|\langle x+2tv\rangle^{-\frac{p}{p+2}s_1} w\|_{L^{q_0}_t H^{1,r_0}_x ( \mathbb{R} \times \mathbb{R}^d)} \lesssim |v|^{-\frac{1}{p+2} }\left \| \varphi \right \|_{H^2 \cap X^{1,s_1}} 
\end{align}
and
\begin{equation} \label{w-scatter}
  \lim_{t \to \pm \infty} \| w(t) - e^{it\Delta} w_\pm \|_{H^1} = 0.
\end{equation}
Furthermore, we have 
\begin{equation}\label{E:supest}
	 \|w - e^{it \Delta }w_- \|_{L^\infty H^1 \cap L^{q_0}H^{1,r_0} } \lesssim |v|^{-\frac{p+1}{p+2}}
	\left \| \varphi \right \|_{H^2 \cap X^{1,s_1}}^{p+1}.
\end{equation}
In particular,
\begin{equation}\label{E:simpleest}
	\|w_+ - w_- \|_{H^1} \lesssim |v|^{-\frac{p+1}{p+2}}
	\left \| \varphi \right \|_{H^2 \cap X^{1,s_1}}^{p+1}.
\end{equation}
\end{thm}

\begin{proof}
Pick $\varphi \in H^2 \cap X^{1,s_1}$.
To solve \eqref{eqn-w},
we set a map
\[
	w \mapsto \Phi(w) = e^{it\Delta} \varphi- i \int_{-\infty}^{t} e^{i(t-s)\Delta} \alpha(x+2sv) |w|^p w(s) \, ds
\]
on a suitable metric space \begin{align*}
    X=\left \{ w:\mathbb{R}\times\mathbb{R}^d\to \mathbb{C} \ \middle|\   \left \| \langle x+2tv\rangle^{-\frac{p}{p+2}s_1}w  \right \|_{L_{t}^{q_0} H_{x}^{1,r_0}} \le 2C_* |v|^{-\frac{1}{p+2} }\left \| \varphi \right \|_{H^2 \cap X^{1,s_1}}   \right \} 
\end{align*}
with the metric
\[
d(u,v)= \left \|\langle x+2tv\rangle^{-\frac{p}{p+2}s_1}[u-v ] \right \|_{L_{t}^{q_0} L_{x}^{r_0}}.
\]
The constant $C_*>0$ encodes implicit constants appearing in several inequalities, including Strichartz's estimates and Sobolev's embedding. 
Our main task is to prove that $\Phi$ is  a contraction map on $X$.

Pick $w\in X$.
For the linear term of $\Phi (w)$, the desired estimate directly follows from the Lemma \ref{weigheted in t,x-free solu}.
Indeed, recalling $s_1p>1$, one has
\begin{equation}\label{E:WPw_pf1}
	\| \langle x+2tv\rangle^{-\frac{p}{p+2}s_1}e^{it\Delta} \varphi  \|_{L_{t}^{q_0} H_{x}^{1,r_0}} \le C_0 |v|^{-\frac{1}{p+2} } \|\varphi\|_{H^2 \cap X^{1,s_1}},
\end{equation}
where $C_0>0$ is a constant.
For the nonlinear term of $\Phi (w)$, we first observe that 
\begin{align*}
&\left\| \langle x+2tv\rangle^{-\frac{p}{p+2}s_1} \int_{-\infty}^{t} e^{i(t-s)\Delta} \alpha(x+2sv) |w|^p w(s) \, ds \right\|_{L^{q_0}_{t}H_x^{1,r_0}} \\
\lesssim &\| \langle x \rangle^{-\frac{p}{p+2}s_1} \|_{W^{1,\infty}_x} \left\| \int_{-\infty}^{t} e^{i(t-s)\Delta} \alpha(x+2sv) |w|^p w(s) \, ds \right\|_{L^{q_0}_{t}H_x^{1,r_0}}.
\end{align*}
Note that $\| \langle x \rangle^{-\frac{p}{p+2}s_1} \|_{W^{1,\infty}_x} \lesssim_{p,d}1$ as $s_1 \in ( \frac{1}{p}, \frac{d}{2}  ) $.
One sees from
 Sobolev's embedding and the interpolation inequality in homogeneous Sobolev spaces that
\[
\left \| w \right \| _{L_x^{r_c}}\lesssim \left \| |\nabla|^{s_c}w \right \|_{L^{r_0}_x}\lesssim \left \| w \right \|  _{H_x^{1,r_0}}
\]holds for $r_c=\frac{dp(p+2)}{4}$ and $ s_c=\frac{d}{2}-\frac{2}{p}\in[0,1].$
Thanks to this inequality and Strichartz's estimate, we have
\begin{equation}\label{E:nonlinearest}
\begin{aligned}
&\left\| \int_{-\infty}^{t} e^{i(t-s)\Delta} \alpha(x+2sv) |w|^p w(s) \, ds \right\|_{L^{q_0}_{t}H_x^{1,r_0}} \\
\lesssim{} & \| \alpha(x+2tv) |w|^p w \|_{L^{q_0'}_t H^{1,r_0'}_x} \\
\lesssim{} &   \| 
\langle x\rangle^{\frac{p+1}{p+2}ps_1}
\alpha \| _{W^{1,\infty}_x}\left \|  \| \langle x+2tv\rangle^{-\frac{p}{p+2}s_1}w  \|_{L_x^{r_c}}^p \| \langle x+2tv\rangle^{-\frac{p}{p+2}s_1}w  \| _{H_x^{1,r_0}}  \right \|_{L_t^{q_0'}}\\
\lesssim{} &
\| 
\langle x\rangle^{\sigma_1}
\alpha \| _{W^{1,\infty}_x}
\| \langle x+2tv\rangle^{-\frac{p}{p+2}s_1} w \| _{L_t^{q_0}H_x^{1,r_0}}^{p+1} .
\end{aligned}
\end{equation}
Hence, together with \eqref{E:WPw_pf1}, one has
\begin{equation}\label{E:WPw_pf2}
	\| \langle x+2tv\rangle^{-\frac{p}{p+2}s_1} \Phi(w) \|_{L_t^{q_0}H_x^{1,r_0}}^{p+1} 
	\le \left(C_0 
	+ C_1 (2C_*)^{p+1} (|v|^{-\frac{1}{p+2} } \|\varphi\|_{H^2 \cap X^{1,s_1}})^p
	\right)|v|^{-\frac{1}{p+2} } \|\varphi\|_{H^2\cap X^{1,s_1}},
\end{equation}
where $C_1>0$ is a constant depending on $d,p$, and $\| 
\langle x\rangle^{\sigma_1}
\alpha \| _{W^{1,\infty}_x}$.
Hence, if we choose $C_*=C_0$
then there exists $C>0$ depending on $d,p$, and $\| 
\langle x\rangle^{\sigma_1}
\alpha \| _{W^{1,\infty}_x}$
such that we have
\[
	C_1 (2C_*)^{p+1} (|v|^{-\frac{1}{p+2} } \|\varphi\|_{H^2 \cap X^{1,s_1}})^p \le C_* \Leftrightarrow |v| \ge C \|\varphi\|_{H^2 \cap X^{1,s_1}}^{p+2}.
\]
If $v$ satisfies this condition then
 $\Phi:X\to X$ holds.

By following the similar estimate as above, we observe that for any $w_1,w_2\in X$, 
\begin{align*}
   d(\Phi(w_1),\Phi(w_2))
&\lesssim \left\| \langle x+2tv\rangle^{-\frac{p}{p+2}s_1}\int_{-\infty}^{t} e^{i(t-s)\Delta} \alpha(x+2sv) [|w_1|^p w_1(s)-|w_2|^pw_2(s)] \, ds \right\|_{L^{q_0}_{t}L_x^{r_0}} \\
&\lesssim  \| \alpha(x+2tv) [|w_1|^p +|w_2|^p] [w_1-w_2] \|_{L^{q_0'}_t L^{r_0'}_x} \\
&\lesssim \| 
\langle x\rangle^{\sigma_1}
\alpha \| _{W^{1,\infty}_x}
|v|^{-\frac{p}{p+2}} \left \| \varphi \right \|_{H^2 \cap X^{1,s_1}}^{p}d(w_1,w_2),
\end{align*}
which implies that, by replacing the constant $C$ with a larger one if necessary, $\Phi$ is a contraction map on $X$ under the same assumption on $|v|$.
Hence, we find a unique solution $w\in X$ to \eqref{eqn-w}.

Let us derive further properties.
By definition of $X$, we have the first property of \eqref{w-property}.
Let us prove the scattering in the positive time direction. In fact, based on the similar estimate as in \eqref{E:nonlinearest}, we obtain 
\begin{align}\label{E:scatter_pf1}
\| e^{-it\Delta} w(t) - e^{-is\Delta} w(s) \|_{H^1}  
\lesssim  \|\langle x+2tv\rangle^{-\frac{p}{p+2}s_1} w \|_{L^{q_0}_t H^{1,r_0}_x\left ( (s,t)\times \mathbb{R}^d \right )}^{p+1}
\end{align}
for any $s<t$.
Since the right hand side tends to zero as $s\to \infty$, there exists a unique $w_+ \in H^1$ such that
\begin{align*}
    \lim_{t \to + \infty} \| e^{-it\Delta}w(t) -  w_+ \|_{H^1}= 0.
\end{align*}
The scattering to $w_-=\varphi$ in the negative time direction can be obtained in a similar way.
Further, by letting $s\to-\infty$ in \eqref{E:scatter_pf1} and using the unitary property of $e^{it\Delta}$, one obtains
\[
	\|w(t)-e^{it\Delta }w_-\|_{H^1} \lesssim 
	 \|\langle x+2tv\rangle^{-\frac{p}{p+2}s_1} w \|_{L^{q_0}_t H^{1,r_0}_x\left ( \R \times \mathbb{R}^d \right )}^{p+1}
	 \lesssim |v|^{-\frac{p+1}{p+2}} \| \varphi\|_{H^2 \cap X^{1,s_1}}^{p+1}
\]
for any $t\in \mathbb{R}$. This is $L^\infty H^1$-part of \eqref{E:supest}. The $L^{q_0} H^{1,r_0}$-part follows by estimating as in \eqref{E:nonlinearest}.
The estimate \eqref{E:simpleest} follows by letting $t \to \infty$ and  $s\to-\infty$ in \eqref{E:scatter_pf1}.
\end{proof}

\subsection{Proof of Theorems \ref{T:well-defined} and \ref{wellposedness-u}}
Now, we are ready to give the proof of Theorem \ref{wellposedness-u}.
Theorem \ref{T:well-defined} directly follows from Theorem \ref{wellposedness-u}.

\begin{proof}[Proof of Theorem \ref{wellposedness-u}]

Recalling the definition of $w$ \eqref{def-w} in terms of $u$, the existence and uniqueness of global solutions $u$ to \eqref{E:NLS} are equivalent to the existence and uniqueness of solutions $w$ to \eqref{eqn-w}.
Moreover, \eqref{def-w} implies that
\[
\left \| \left \langle x \right \rangle ^{-\frac{p}{p+2}s_1 }u(t)  \right \| _{H^{1,r_0}_x}\lesssim \left \langle v \right \rangle \left \| \left \langle x+2tv \right \rangle^{-\frac{p}{p+2}s_1 }w(t)  \right \|_{H^{1,r_0}_x} .
\]
Taking the $L_t^{q_0}$-norm of both sides and using the estimate \eqref{w-property}, we obtain
\[
 \|\langle x\rangle^{-\frac{p}{p+2}s_1} u\|_{L^{q_0}_t H^{1,r_0}_x ( \mathbb{R} \times \mathbb{R}^d)} \lesssim |v|^{\frac{p+1}{p+2} }\left \| \varphi \right \|_{H^2 \cap X^{1,s_1}}
\]
if $|v| \ge C\| \varphi \|_{H^2 \cap X^{1,s_1}}^{p+2} +1$.

In addition, once \(v\) is fixed, it follows from \eqref{GT}, \eqref{def-w} and \eqref{w-scatter} that
\[
\lim_{t\to \pm\infty}
\left\|u(t)-e^{it\Delta}e^{iv\cdot x}w_{\pm}\right\|_{H^1}
=
\lim_{t\to \pm\infty}
\left\|e^{iv\cdot x}\left [ w(t)-e^{it\Delta}w_{\pm} \right ]  \right\|_{H^1}
=0.
\]
This completes the proof of Theorem \ref{wellposedness-u}.
\end{proof}

\section{The Inverse Problem}\label{sec-inverse}

Our goal is to show that knowledge of $S$ is sufficient to determine the nonlinearity in \eqref{E:NLS}.

\subsection{A key convergence theorem}

In this section, we prove Theorem \ref{T:inverse-formula}.
In particular, we derive the following formula
\[
\lim_{\rho\to\infty}i\rho\left \langle (S-I)(e^{i\rho\theta \cdot x}\varphi) ,e^{i\rho\theta \cdot x}\psi \right \rangle=\int_{\mathbb{R} }\left \langle \alpha(x+2t\theta)|\varphi|^p \varphi,\psi\right \rangle  dt ,\quad\theta\in \mathbb{S}^{d-1}.
\] 
Note that a simple estimate \eqref{E:simpleest} is not sufficient to obtain this limit.

\begin{proof}[Proof of Theorem \ref{T:inverse-formula}]
Given a fixed $\theta\in \mathbb{S}^{d-1}$, we set $v=\rho\theta$ and $u_{-}=e^{iv\cdot x}\varphi$.
Thanks to Theorem \ref{T:well-defined}, if $\rho$ is large then $u_- \in B$ and $Su_-$ makes sense. 
For such $\rho$, by
definition of scattering map $S: u_{-}\mapsto u_{+}$ and $w(t)$, we have
\begin{align*}
&i\left \langle (S-I)e^{iv\cdot x}\varphi, e^{iv\cdot x}\psi \right \rangle \\
=&i\left [ \lim_{t \to +\infty}\left \langle e^{-it\Delta}u(t),e^{iv\cdot x}\psi  \right \rangle-\left \langle e^{iv\cdot x}\varphi, e^{iv\cdot x}\psi \right \rangle   \right ] \\
=&i\left [ \lim_{t \to +\infty}\left \langle w(t), e^{it\Delta}\psi \right \rangle-\left \langle \varphi,\psi \right \rangle   \right ]\\
=&i \lim_{t \to +\infty}\left \langle e^{-it\Delta}w(t)-\varphi,\psi \right \rangle\\
=&\int_{\mathbb{R}}\left \langle \alpha(x+ 2tv)|w|^p w(t), e^{it\Delta}\psi \right \rangle \ dt.
\end{align*}
To consider the limit $\rho\to \infty$, we split it into the main term and error term as follows:
\begin{align*}
&\int_{\mathbb{R}}\left \langle \alpha(x+ 2tv)|w|^p w(t), e^{it\Delta}\psi \right \rangle \ dt
=M+E,
\end{align*}
where
\begin{equation}\label{E:defM}
	M:= \int_{\mathbb{R}}\left \langle \alpha(x+2tv)\left | e^{it\Delta}\varphi \right |^p e^{it\Delta}\varphi, e^{it\Delta}\psi  \right \rangle\ dt
\end{equation}
and
\begin{equation}\label{E:defE}
	E:= \int_{\mathbb{R}}\left \langle \alpha(x+2tv)\left [ |w|^p w(t)- \left | e^{it\Delta}\varphi \right |^p e^{it\Delta}\varphi\right ] , e^{it\Delta}\psi\right \rangle\ dt.
\end{equation}

First, we claim that \begin{align}\label{claim 1}
    \lim_{\rho\to\infty}\rho M=\int_{\mathbb{R} }\left \langle \alpha(x+2t\theta)|\varphi|^p \varphi,\psi\right \rangle  dt. 
\end{align}
Indeed, by using (\ref{GT}) and a change of variables, we derive\begin{align*}
    M={}&\int_{\mathbb{R}}\left \langle \alpha \left | e^{it\Delta}\varphi(x-2\rho\theta t) \right |^p e^{it\Delta}\varphi(x-2\rho\theta t),e^{it\Delta}\psi(x-2\rho\theta t)  \right \rangle dt\\
={}&\frac{1}{\rho}  \int_{\mathbb{R}}\left \langle \alpha(x+2\theta t)\left | e^{i\frac{t}{\rho}\Delta }\varphi \right |^pe^{i\frac{t}{\rho}\Delta }\varphi,e^{i\frac{t}{\rho}\Delta }\psi  \right \rangle dt .
\end{align*}
Denote
\begin{align*}
    g_{\rho}(t)&:=\left \langle \alpha(x+2\theta t)\left | e^{i\frac{t}{\rho}\Delta }\varphi \right |^pe^{i\frac{t}{\rho}\Delta }\varphi,e^{i\frac{t}{\rho}\Delta }\psi \right \rangle,\\
h(t)&:=\left \langle \alpha(x+2\theta t)\left | \varphi \right |^p \varphi,\psi \right \rangle.  
\end{align*}
Then, by using the dominated convergence theorem, what we need to prove is that 
\begin{align}
    \lim_{\rho\to\infty}g_{\rho}(t)&=h(t),
    \label{claim1-1}\\
    \left | g_{\rho}(t) \right |&\lesssim \langle t \rangle^{-\sigma_2} \in L^1_t \label{claim1-2}
\end{align} hold for any $t\in \mathbb{R}$.

To achieve this goal, we first observe that 
\[
	 \left | g_{\rho}(t)-h(t) \right |  
\le \mathrm{I} + \mathrm{II},
\]
where
\begin{align*}
  \mathrm{I} :={}&\left | \left \langle  \alpha(x+2\theta t)\left [ \left | e^{i\frac{t}{\rho}\Delta }\varphi \right |^p e^{i\frac{t}{\rho}\Delta }\varphi-\left | \varphi \right |^p \varphi   \right ] ,e^{i\frac{t}{\rho}\Delta }\psi  \right \rangle  \right |, 
  \\
\mathrm{II} :={}&\left | \left \langle \alpha(x+2\theta t)\left | \varphi \right |^p\varphi, e^{i\frac{t}{\rho}\Delta }\psi-\psi \right \rangle  \right | .
\end{align*}
Let us estimate $\mathrm{I}$.
By using Sobolev's embedding $\| \varphi  \| _{L^{p+2}}\le \| \varphi \| _{H^1}$ and the fact that
 \[
	\lim_{\tau \to 0} \| (e^{i\tau \Delta} -1)\varphi \|_{L^2} = 0
\]
for any $\varphi \in L^2$ follows by the dominated convergence theorem,
 we have \begin{align*}
    \mathrm{I}&\lesssim \left \| \alpha \right \| _{L^{\infty}}\left \{  \left \| e^{i\frac{t}{\rho}\Delta} \varphi\right \| _{L^{p+2}}^p +\left \| \varphi \right \| _{L^{p+2}}^p\right \}\left \| e^{i\frac{t}{\rho}\Delta} \varphi-\varphi \right \|_{L^{p+2}}\left \| e^{i\frac{t}{\rho}\Delta} \psi \right \|_{L^{p+2}}\\
&\lesssim \left \| \alpha \right \| _{L^{\infty}}\left \| \varphi \right \| _{H^{1}}^p\left \| (e^{i\frac{t}{\rho}\Delta}-1)|\nabla|^{\frac{pd}{2(p+2)}} \varphi  \right \| _{L^2}\left \| \psi \right \|_{L^2} \to 0 
\end{align*}
as $\rho\to\infty$.
Notice that $\frac{pd}{2(p+2)} \in (0,1]$ by assumption on $p$ and hence 
 $|\nabla|^{\frac{pd}{2(p+2)}} \varphi  \in L^2$, provided $\varphi \in H^1$.
One can handle $\mathrm{II}$ in the same way. Hence, we obtain \eqref{claim1-1}.

Next, let us prove \eqref{claim1-2}.
We use the assumption $(\mathrm{A2})$ for this part.
One has
\begin{align*}
    |g_{\rho}(t)|&=\left |\left \langle \alpha(x+2\theta t)\left | e^{i \frac{t}{\rho} \Delta}  \varphi \right |^p e^{i \frac{t}{\rho} \Delta}  \varphi,e^{i \frac{t}{\rho} \Delta}  \psi  \right \rangle\right | \\
&\lesssim \left \| \alpha(x+2\tfrac{ t}\rho v) e^{i \frac{t}{\rho} \Delta}\varphi\right \|_{L^2} \left \| e^{i \frac{t}{\rho} \Delta}\varphi \right \|_{L^{2(p+1)}}^{p} \| e^{i \frac{t}{\rho} \Delta}\psi \|_{L^{2(p+1)}}  \\
&\lesssim \left \langle t\theta \right \rangle ^{-\sigma_2}\left \| \varphi \right \| _{X^{s_2}}\left \| \varphi \right \| _{H^{2}}^{p} \| \psi \|_{H^{2}} \\
&\lesssim \left \langle t \right \rangle ^{-\sigma_2},
\end{align*} where we have used Sobolev embedding $H^2 \hookrightarrow H^{\frac{2d}{d+2}} \hookrightarrow L^{2(p+1)}$
and Proposition \ref{weighted estimate-free solu} to obtain the third line.
This completes the proof of \eqref{claim1-2}.

It remains to prove that \[
|E|=o(\rho^{-1}) \quad \text{as} \quad \rho\to\infty.
\]
By using the equation of $w$, \eqref{E:supest} and estimating as in \eqref{E:nonlinearest}, we have 
\begin{align*}
|E|
\lesssim {}&\left \| \alpha(x+2tv)e^{it\Delta}\psi\left [  |w|^p+ \left | e^{it\Delta}\varphi \right |^p \right ] \left ( w- e^{it\Delta}\varphi\right )  \right \|_{L_{t,x}^1}\\
\lesssim {}&\left \| \alpha(x+2tv) e^{it\Delta}\psi\left [  |w|^p+ \left | e^{it\Delta}\varphi \right |^p\right ]  \right\|_{L_t^{q_0'}L_x^{r_0'}}\left \| w- e^{it\Delta}\varphi\right \|_{L_t^{q_0}L_x^{r_0}} \\
\lesssim {}& |v|^{-\frac{p+1}{p+2}} \left \| \alpha(x+2tv) e^{it\Delta}\psi\left [  |w|^p+ \left | e^{it\Delta}\varphi \right |^p\right ]  \right\|_{L_t^{q_0'}L_x^{r_0'}} 
 \\
\lesssim {}& |v|^{-\frac{p+1}{p+2}} \left [  \left \| \langle x+2tv\rangle^{-\frac{p}{p+2}s_1 } w\right \| _{L_t^{q_0}H_x^{1,r_0}}^p+ \left \| \langle x+2tv\rangle ^{-\frac{p}{p+2}s_1}e^{it\Delta} \varphi \right \| _{L_t^{q_0}H_x^{1,r_0}}^p\right ] \\
& \qquad \times \left \| \langle x+2tv\rangle^{-\frac{p}{p+2}s_1}e^{it\Delta} \psi \right \| _{L_t^{q_0}L_x^{r_0}}\\
\lesssim {}&|v|^{-\frac{2(p+1)}{p+2}},
\end{align*}
which is $o(|v|^{-1})$ as $ |v|\to\infty$. This completes the proof of \eqref{E:recovery}.
\end{proof}

\subsection{Recovery formulae}

We finally prove Theorem \ref{T:recoveryf}. The reconstruction argument follows the standard strategy developed in previous inverse scattering studies \cite{Wa,Mu}.
We give a proof for the sake of completeness.

\begin{proof}[Proof of Theorem \ref{T:recoveryf}]
Let us begin with the proof of the formula \eqref{E:precov}.
Suppose that $\varphi \in C_0^\infty$ and $\theta \in S^{d-1}$ satisfy
\[
	\lim_{\rho\to\infty}
	i\rho
\left\langle
(S-I)(e^{i\rho\theta\cdot x}\varphi),
e^{i\rho\theta\cdot x}\varphi
\right\rangle
=
\int_{\R} \int_{\R^d}
\alpha(x+2t\theta)
|\varphi(x)|^{p+2} dx dt\neq0.
\]
Note that the equality is due to Theorem \ref{T:inverse-formula}.
Then, we also have
\[
	\lim_{\rho\to\infty}
	i\rho
\left\langle
(S-I)(e^{1+i\rho\theta\cdot x}\varphi),
e^{i\rho\theta\cdot x}\varphi
\right\rangle
=
e^{p+1}
\int_{\R} \int_{\R^d}
\alpha(x+2t\theta)
|\varphi(x)|^{p+2} dx dt.
\]
and hence
\[
\frac{\langle (S - I)(e^{1+i\rho \theta \cdot x} \varphi), e^{i\rho \theta \cdot x} \varphi \rangle}{\langle (S - I)(e^{i\rho \theta \cdot x} \varphi), e^{i\rho \theta \cdot x} \varphi \rangle} \to e^{p+1} \quad \text{as } \rho \to \infty.
\]
This shows \eqref{E:precov}.

Let us proceed to the proof of \eqref{E:arecov}.
Suppose that $\varphi$ is compactly supported and normalized as $\|\varphi\|_{L^{p+2}}=1$.
Fix $\theta \in S^{d-1}$ and $y \in \mathbb{R}^d$.
We see from Theorem \ref{T:inverse-formula} that
\[
\lim_{\rho\to\infty}
	i\rho
\left\langle
(S-I)(e^{i\rho\theta\cdot x}
\varphi_n(x-y),
e^{i\rho\theta\cdot x}\varphi_n(x-y)
\right\rangle=
\int_{\mathbb{R}} \int_{\mathbb{R}^d} \alpha(x+2\theta t) |\varphi_n(x-y)|^{p+2} \, dx \, dt 
\]
for each $n\ge1$.
Denote
\[
g_n(t) = \int_{\mathbb{R}^d} \alpha(x+2\theta t) |\varphi_n(x-y)|^{p+2} \, dx.
\]
Then, noting that $\int |\varphi_n|^{p+2}dx =1$, we obtain
\begin{align*}
    \left | g_n(t)-\alpha(y+2\theta t) \right | \le {}&\int_{\mathbb{R}^d}\left | \alpha(x+2\theta t)|\varphi_n(x-y)|^{p+2}-\alpha(y+2\theta t)|\varphi_n(x-y)|^{p+2} \right | dx\\
={}& \int_{\mathbb{R}^d}\left | \alpha(x/n+y+2\theta t)-\alpha(y+2\theta t)\right | |\varphi(x)|^{p+2}  dx \\ \to{}& 0 
\end{align*}
as $n\to\infty$ by dominated convergence theorem.
Hence,
\[
g_n(t) \to \alpha(y + 2\theta t) \quad \text{as } n \to \infty
\]
 for all $ t \in \mathbb{R}$.

Let $L>0$ be the number such that $\supp \varphi \subset \{|x| \le L \}$.
Then, by assumption $(\mathrm{A2})$, we have
\begin{align*}
    |g_n(t)|&\lesssim \int_{\mathbb{R}^d}\left \langle x+2\theta t \right \rangle^{-\sigma_2}|\varphi_n(x-y)|^{p+2} \ dx\\
&\lesssim \int_{\{|x|\le L\}}\left \langle x+y+2\theta t \right \rangle^{-\sigma_2}|\varphi_n(x)|^{p+2} \ dx\\
&\le \int_{\mathbb{R}^d} h(t)|\varphi_n(x-y)|^{p+2}\ dx
= h(t)
\end{align*}
for $n\ge1$,
where
\begin{equation*}
    h(t):= \sup_{|x|\le L}\langle  x+y+2\theta t \rangle^{-\sigma_2}
    \le \langle  \max(2|t| - (|y|+L),0) \rangle^{-\sigma_2}
    \in L_t^1(\mathbb{R} ).
\end{equation*}
Now, by the dominated convergence theorem, we conclude that
\[
\lim_{n\to\infty} \int_\R g_n(t) dt=
\int_{\mathbb{R}} \alpha(y + 2\theta t) \, dt .
\]
This completes the proof of \eqref{E:arecov}.
\end{proof}

\subsection*{Acknowledgements} 
S.M. was  supported by JSPS KAKENHI Grant Numbers JP23K20803, JP23K20805, and JP24K00529.

\end{document}